\documentclass[a4paper,12pt]{article}

\usepackage[intlimits]{amsmath}
\usepackage{amsthm,amssymb}
\usepackage{graphicx}
\usepackage{dsfont}
\usepackage{enumerate}
\usepackage{mathrsfs}
\usepackage[colorlinks]{hyperref}

{\theoremstyle{plain}
\newtheorem{thm}{Theorem}[section]
\newtheorem{lem}[thm]{Lemma}

}

{\theoremstyle{definition}
\newtheorem{ex}{Example}
}

\numberwithin{equation}{section}

\usepackage{graphicx}
\graphicspath{{Fig/}}
\usepackage{pgf,tikz}
\usetikzlibrary{arrows}
\usetikzlibrary{intersections}

\newcommand{\vect}{}

\newcommand{\vvectu}{u_{0:d}}
\newcommand{\vvectuone}{u_{1:d}}
\newcommand{\cent}{\mathfrak{c}}

\newcommand{\origin}{\vect{o}}

\newcommand{\B}{B^d}

\newcommand{\simplex}{T}

\newcommand{\dint}{\mathrm{d}}

\newcommand{\E}{\mathbb{E}}
\newcommand{\N}{\mathbb{N}}
\renewcommand{\P}{\mathbb{P}}
\newcommand{\Q}{\mathbb{Q}}
\newcommand{\R}{\mathbb{R}}
\renewcommand{\S}{\mathbb{S}}

\newcommand{\Fc}{\mathcal{F}}
\newcommand{\Hc}{\mathcal{H}}
\newcommand{\Kc}{\mathcal{K}}

\newcommand{\Pc}{\mathcal{P}}
\newcommand{\Pcn}{\Pc_n}

\newcommand{\ck}{\mathfrak{c}}

\newcommand{\sk}{\mathfrak{s}}

\newcommand{\1}{\mathds{1}}

\newcommand{\radius}{\mathtt{r}}
\newcommand{\Radius}{\mathtt{R}}

\newcommand{\Pstraight}{\textnormal{\textsf{P}}}

\newcommand{\Pcc}{\Pc_\cent}

\newcommand{\Pns}{\Pc_{n,\sk}}

\newcommand{\Kcs}{\Kc_\sk} 
\newcommand{\Ks}{\Kc_\sk}

\newcommand{\IntMos}{\gamma^{(d)}}

\newcommand{\nmin}{n_{\min}}

\begin{document}
  
  \title{Small cells in a Poisson hyperplane tessellation}
  
  \author{
    Gilles Bonnet\thanks{
      Ruhr Universit\"at Bochum, 
      Email: gilles.bonnet@rub.de} }
  
  \date{April, 2018}
  
  \maketitle
  
  \begin{abstract}
    Until now, little was known about properties of small cells in a Poisson hyperplane tessellation.
    The few existing results were either heuristic or applying only to the two dimensional case and for very specific size functionals and directional distributions.
    This paper fills this gap by providing a systematic study of small cells in a Poisson hyperplane tessellation of arbitrary dimension, arbitrary directional distribution $\varphi$ and with respect to an arbitrary size functional $\Sigma$.
    More precisely, we investigate the distribution of the typical cell $Z$,  conditioned on the event $\{\Sigma(Z)<a\}$, where $a\to0$ and $\Sigma$ is a size functional, i.e.\ a functional on the set of convex bodies which is continuous, not identically zero, homogeneous of degree $k>0$, and increasing with respect to set inclusion.
    We focus on the number of facets and the shape of such small cells.
    We show in various general settings that small cells tend to minimize the number of facets and that they have a non degenerated limit shape distribution which depends on the size $\Sigma$ and the directional distribution.
    We also exhibit a class of directional distribution for which cells with small inradius do not tend to minimize the number of facets.
  \end{abstract}
  
  \section{Introduction}
    D.G. Kendall recalled in 1987 in the foreword of \cite{StoyanKendallMecke87} a conjecture he made a few decades before about the shape of big cells in random tessellations. 
    He considered a planar stationary isotropic Poisson line tessellation and conjectured that the zero cell $Z_{\origin}$ (i.e.\ the cell containing the origin) tends do be circular, if we condition on its area $V_2(Z_{\origin})\to\infty$.
    This conjecture was later proved by Kovalenko \cite{Kovalenko1997,Kovalenko1999} and many contributions to this problem and very broad generalisations of it have been done by Miles, Goldman, Mecke, Osburg, Hug, Reitzner and Schneider.
    See Note 9 of Section 10.4 in \cite{SchneiderWeil08} for precise references.
  
    In contrast to D.G.\ Kendall's problem we are interested in the shape of small cells for which much less is known.
    Let $\eta$ be a stationary Poisson hyperplane process in $\R^d$, with $d\geq2$, of intensity measure 
    \begin{equation*}
      \gamma  \mu (\cdot)
      :=
       \gamma \int\limits_{\S^{d-1}}\int\limits_0^\infty \1 \left( H( \vect{u} ,t)\in\cdot \right) \dint t \, \dint \varphi( \vect{u} ) \,, 
    \end{equation*}
    where $\gamma>0$, $ H( \vect{u} ,t) = \{ \vect x \in \R^d : \langle \vect x , \vect u \rangle = t \} $ is the hyperplane orthogonal to $\vect u$ at distance $t$ from the origin $\origin$, and $\varphi$ is an even probability measure on $\S^{d-1}$ whose support is not contained in some great circle.
    We call $\gamma$ the intensity and $\varphi$ the directional distribution.
    We consider the typical cell $Z$ of the corresponding hyperplane tessellation.
    It will be defined in Section \ref{sec:Setting}, but we give here an intuitive description of it.
    Consider a window $t W\subset \R^d$, with $W$ a convex body containing the origin and $t>0$ large.
    Among the cells which intersect $t W$, choose one uniformly at random and translate it to the origin. 
    This random polytope converges in distribution to $Z$, as $t\to\infty$.
    We consider the following random variables:
    \begin{itemize}
      \item $f(Z)$, the \textit{number of facets} i.e.\ $(d-1)$-dimensional faces, of $Z$,
      \item $\sk(Z)$, the \textit{shape} of $Z$, which can be viewed as the congruence class of $Z$ under the scale action,
      \item $\Sigma(Z)$, the \textit{size} of $Z$, where $\Sigma$ is any real function on the set of convex bodies (convex and compact sets with non empty interior) which is
                 continuous,
                 not  identically zero,
                 homogeneous of some degree $k>0$,
                 and increasing under set inclusion ($K\subset L \Rightarrow \Sigma(K) \leq \Sigma(L)$).
    \end{itemize}
    
    Let $\nmin := {\min} \{ n\in\N : \P(f(Z)=n)>0  \}$.
    We are interested by the following questions.
    \begin{enumerate}
      \item[(Q1)] Is it true that $\P(f(Z)=\nmin \mid \Sigma(Z)<a) \to 1 $, as $a\to0$?
      \item[(Q1')] If yes, what is the speed of convergence?
      \item[(Q2)] What is the limiting distribution (if well defined) of the random variable $\sk(Z)$ conditioned on the event $\{\Sigma(Z)<a\}$, as $a\to0$?
    \end{enumerate}
    The answers to these questions turn out to strongly depend on $\Sigma$ and on the directional distribution $\varphi$ of $\eta$.
    
    
    In \cite{Miles95}, Miles considered the planar isotropic case, i.e.\ $\varphi$ is the normalized spherical Lebesgue measure.
    He gave heuristic arguments that the shape of $Z$ conditioned on $\Sigma(Z)\to0$ is a triangle with a random shape depending on $\Sigma$.
    
    In \cite{BeermannRedenbachThaele14}, Beermann, Redenbach and Th{\"a}le  considered the planar case where the support of $\varphi$ is concentrated in two couples of antipodal points. 
    In this situation all cells are parallelograms.
    This setting is particularly interesting because the behavior of small cells can change drastically depending on which size functional $\Sigma$ is considered. 
    They show that, when $\Sigma$ is the perimeter, the shape of $Z$ conditioned on $\Sigma(Z)\to0$ converges weakly to a random non-degenerated parallelogram.
    Moreover, if the two directions are distributed with equal probability, it even holds that the shape of $Z$ is independent from its perimeter.
    In contrast to these results, they also show that when $\Sigma$ is the area, the shape of $Z$ conditioned on $\Sigma(Z)\to0$ converges weakly to the shape of a random line segment, i.e.\ a degenerated parallelogram.
     
    We also want to mention that Schulte and Th\"ale in \cite[Cor. 5]{SchulteThaele16}, and Chenavier and Hemsley in \cite[Thm. 2]{ChenavierHemsley15}, both gave results, in the stationary and isotropic planar case, about the smallest cell(s) in a window of increasing size.
    Schulte and Th\"ale show that the area of the smallest triangular cell (with respect to the area) after a proper rescaling converges in distribution to a Weibull distributed random variable.
    Chenavier and Hemsley show that, for any $k\in \N$, the $k$ smallest cells (with respect to the inradius) are triangles with high probability when the window is big enough.
    
    It seems that apart from the results mentioned  above, nothing else is known about small cells.
    The present paper fills this gap by a systematic study of the small typical cell for general dimensions, directional distributions and size functionals.
    
    \bigskip
    
    In the next sections we will present general results.
    Before that, we give examples for which our results apply.
    The set of convex bodies in $\R^d$ is denoted by $\Kc$.
    For any $K\in\Kc$, we denote by $\radius(K)$ the biggest radius possible of a closed ball contained in $K$, and by $\Radius(K)$ the radius of the smallest ball containing $K$.  
    These two size functionals are of particular interest, since they give the following bounds for any size functionals $\Sigma$ of degree $k>0$,
    \begin{equation} \label{ineq:sizes}
       c \, \radius(K) 
       \leq \Sigma(K)^{\frac1k}
       \leq c \, \Radius(K) ,
       \text{ for any convex body } K, 
    \end{equation}
    where $c = \Sigma(\B)^{\frac1k}$.
    Therefore, for any Borel set $A\subset\Kc$, and $a>0$,
    \[
      \P\left( Z \in A \,,\, \Radius(Z) <  \frac{a}{c} \right)
      \leq \P\left( Z \in A \,,\, \Sigma(Z)^{\frac1k} < a \right)
      \leq \P\left( Z \in A \,,\, \radius(Z) < \frac{a}{c} \right) \,.
    \]
    
    \begin{ex}[general $\varphi$, $\Sigma^{\frac1k}\geq c \, \Radius$]
      Assume that $\Sigma^{\frac1k} \geq c \, \Radius$ for some constant $c>0$.
      Several classical size functionals satisfy this condition.
      For example: the diameter (maximal distance between two points in a convex body), the mean-width, 
      or $\Radius$. 
      We will show that
      \[
        \P\left(f(Z)=\nmin  \mid \Sigma(Z)^{\frac1k}<a\right) \to 1, \text{ as } a\to0.
      \]      
      In the trivial case where the support of $\varphi$ consists of $d$ pairs of antipodal points, all cells are parallelepipeds and therefore  
      $$ \P\left(f(Z)=\nmin  \mid \Sigma(Z)^{\frac1k}<a\right) 
      = \P\left(f(Z)=2d \right)  
      =1 \,,$$
      for any $a>0$.
      Apart from this case, we have a linear speed of convergence, that is
      \[
        \P\left(f(Z)>\nmin  \mid \Sigma(Z)^{\frac1k}<a\right) 
        \sim c' \gamma a , \text{ as } a\to0\,,
      \]
      where $c'$ is a positive constant depending only on $\varphi$ and $\Sigma$.
      Theorem \ref{thm:generalvarphi} presents a refined version of these results.
      In particular it gives the limiting distribution of $\sk(Z)$ conditioned on $\{\Sigma(Z)^{\frac1k}<a\}$.
    \end{ex}
    
    \begin{ex}[absolutely continuous $\varphi$, general $\Sigma$]
      Assume that $\varphi$ is absolutely continuous with respect to the spherical Lebesgue measure.
      This setting includes the  isotropic case.
      Theorem \ref{them:abscont1} will tell us that, for any size functional $\Sigma$ of degree $k>0$, namely
      \[
        \P\left(f(Z)=d+1 \mid \Sigma(Z)^{\frac1k}<a\right) \to 1, \text{ as } a\to0\,,
      \]
      moreover it also describes the limiting shape distribution.
      And Theorem \ref{thm:speed} gives bounds for the speed of convergence,
      \[
        c \gamma a
        \leq \P\left(f(Z)>d+1 \mid \Sigma(Z)^{\frac1k}<a\right) 
        \leq c' \gamma a \ln\left(\frac{1}{\gamma a}\right)\,,
      \]
      where the positive constants $c$ and $c'$ depend on $\Sigma$, and where the lower bound is sharp if $\Sigma=\Radius$.
      We conjecture that the upper bound is sharp when $\Sigma=\radius$.
    \end{ex}
        
    \begin{ex}[$\varphi$ with atoms, $\Sigma=\radius$]
      Finally we exhibit a class of directional distribution for which $f(Z)$ does not converge in distribution to $\nmin$ when conditioned on $\{\radius(Z)<a\}$ with $a\to0$.
     
      We assume that $\varphi$ has atoms, i.e.\ there exists $u\in\S^{d-1}$ such that $\varphi(\{u,-u\})>0$, and that the support of $\varphi$ includes $d+1$ distinct points which are not all contained in some half sphere.
      Then, Theorem \ref{them:atoms} says that
      \[
        \P\left(f(Z)=\nmin  \mid \radius(Z)<a\right) \not\to 1, \text{ as } a\to0 \,.
      \] 
    \end{ex}
    
    \begin{center}
      \begin{table}
        \label{table:examples}
        \[
          \begin{array}{|c||c|c|c|}
            \hline 
            \Sigma \quad\backslash\quad \varphi & \text{general } \varphi & \text{isotropy/abs. cont.} & \text{with atom \& condition on $\mathrm{supp}\varphi$}\\ 
            \hline \hline
            \text{general } \Sigma  &  & f(Z)\xrightarrow{d} d+1 &  \\
            \hline 
            \Radius & f(Z)\xrightarrow{d} \nmin  & f(Z)\xrightarrow{d} d+1  & f(Z)\xrightarrow{d}  \nmin \\
            \hline 
            \radius &  & f(Z)\xrightarrow{d} d+1 & f(Z)\not\xrightarrow{d} \nmin \\
            \hline 
          \end{array} 
        \]
        \caption{Answers to (Q1) depending on $\Sigma$ and $\varphi$. 
        In the right column, $\varphi$ is such that it has atoms and its support includes $d+1$ distinct points which are not all contained in some half sphere.}
      \end{table}
    \end{center}     
    \bigskip 
  
  Our paper is structured as follow.
  In Section \ref{sec:Setting} we set up the notation, define formally the typical cell and present the well known Complementary Theorem.
  In the remaining $3$ sections we study the asymptotic of small cells with respect to different conditions on the directional distribution $\varphi$. 
  In Section \ref{sec:generalvarphi} we provide results which apply for general $\varphi$.
  In Section \ref{sec:abscont} we consider the case where $\varphi$ is absolutely continuous, while Section $\ref{sec:atoms}$ focuses on a specific class of directional distributions with atoms. 
     
  \section{Preliminaries} \label{sec:Setting}
    We work in a $d$-dimensional Euclidean vector space $\R^d$, $d\geq2$, with scalar product $\langle\cdot,\cdot\rangle$, norm $\lVert\cdot\rVert$ and origin $ \origin $.
    We denote by $ B ( \vect{x} , r ) $ the closed ball and by $ S ( \vect{x} , r ) = \partial B ( \vect{x} , r ) $ the sphere with center $ \vect{x} $ and radius $ r $, by  $ \B = B ( \origin , 1 ) $ the unit ball and  by $ \S^{d-1} = \partial \B $ the unit sphere.
    Let $\Hc $ be the space of affine hyperplanes in $\R^d$ with its usual topology and Borel structure, see \cite{SchneiderWeil08}.
    Every hyperplane $H\in\Hc $, which does not contain the origin, has a unique representation 
    \[
      H( \vect{u} ,t)
      :=\{ \vect{x}\in\R^d :\  \langle \vect{x}, \vect{u} \rangle=t\} \,,
    \]
    with $ \vect{u} \in \S^{d-1}$ and $t>0$.
    For such hyperplane, we denote by $H^-$ (resp. $H^+$) the halfspace supported by $H$ containing the origin (resp. not containing the origin).
    
    
    Let $\Kc $ be the set of convex bodies (compact convex sets of $\R^d$ with non-empty interior).
    We write $\Pc $ for the set of all polytopes, and $f(P)$ for the number of facets of a polytope $P \in \Pc$.
    Denote by $\Pcn =\{P \in \Pc : f(P)=n \} $ the set of $n$-topes, hence $\Pcn \subset \Pc \subset\Kc$.
    For any $ t > 0 $, and $ K , L \in \Kc $ we define
    \[
      t K := \{ t \vect{x} :\  \vect{x} \in K \},
      \quad
      K + L := \{ \vect{x} + \vect{y} :\  \vect{x} \in K , \vect{y} \in L \} \,,
    \]
    where the latter is the Minkowski sum of $K$ and $L$.
    
    The sets $\Kc$, $\Pc$, and $\Pcn$ are equipped with the Hausdorff distance $d_H$, 
    \[
      d_H(K,L) = \min \{r \geq 0 :  K\subset L+rB^d,\ L \subset K+rB^d \} \,,
    \]
    and with the associated topology and Borel structure.

    As in \cite{HugSchneider2007}, a functional $\Sigma:\Kc\to\R$ is called size functional of degree $k>0$ if it is 
       continuous,       
       not  identically zero,
       homogeneous of some degree $k$ meaning that $\Sigma(tK)=t^k \Sigma(K)$,
       and increasing under set inclusion ($K\subset L \Rightarrow \Sigma(K) \leq \Sigma(L)$).
       Note that this definition implies that $ \Sigma $ is positive.
      For any $K\in\Kc$, we denote by $\radius(K)$ the biggest radius possible of a closed ball contained in $K$, and by $\Radius(K)$ the radius of the smallest ball containing $K$.
      These two size functionals are called inradius and circumradius respectively.

    Let $\eta$ be a stationary Poisson hyperplane process in $\R^d$, that is a Poisson point process in the space $\Hc $ whose intensity measure is of the form 
    \begin{equation}\label{def:Thetamu}
      \E \, \eta ( \cdot )
      = \gamma  \mu (\cdot)
      := \gamma \int\limits_{\S^{d-1}}\int\limits_0^\infty \1 \left( H( \vect{u} ,t)\in\cdot \right) \dint t \, \dint \varphi( \vect{u} ) \,, 
    \end{equation}
    where $ \gamma >0$, $\mu$ is a measure  on $\Hc $, $\varphi$ is an even probability measure on $\S^{d-1}$ and $\1(\cdot)$ is the indicator function.
    We call $ \gamma $ the {intensity} and $\varphi$ the \textit{directional distribution} of the hyperplane process $\eta$.
    We assume that the support of $\varphi$ is not contained in a great circle of $\S^{d-1}$.
    When $\varphi$ is the normalized spherical Lebesgue measure on $\S^{d-1}$, we say that $\eta$ is {isotropic}.
    
    The closure of each of the connected components of the complement of the union $\bigcup_{H\in\eta} H$ is almost surely a polytope (because the support of $\varphi$ is not contained in a great circle).
    We denote by $X$ the collection of these polytopes.
    It is called the Poisson hyperplane tessellation induced by $\eta$, and the polytopes are called the cells of $X$.
    We can see $X$ as a point process in $\Pc$.
    To describe the distribution of $X$ we need the notion of the \textit{$\Phi$-content} of a convex body $K$ which is given by
    \[
      \Phi(K)
      := \int\limits_{\S^{d-1}} h(K, \vect{u} ) \,\dint \varphi(\vect{u} ) \,,
    \]
    where $h(K, \vect{u} ):=\max\{\langle \vect{x}, \vect{u} \rangle:\  \vect{x}\in K\}$ is the value of the {support function} of $K$ at $ \vect{u} $.
%
    The definition of $\Phi$ is motivated by the fact that, for any $K\in\Kc $,
    \[
      \P (\eta\cap K=\emptyset)
      = e^{- \gamma \Phi(K)} \,,
    \]
    where $ \eta $ is identified with the union $ \cup_{H\in\eta} H $.
    
    Note that $\Phi$ is homogeneous of degree $1$ and translation invariant, meaning that $\Phi( t K + \vect{x} )  = t \Phi(K)$ for any $ K \in \Kc $, $t \geq 0 $ and $ \vect{x} \in \R^d $.
In the isotropic case, the $\Phi$-content of a convex set $K$ is half of the mean-width and up to a constant the first intrinsic volume $V_1(K)$.
    Note that $K \in \Kc$ implies that  $\Phi(K)>0$ since $K$ contains at least $2$ points.
    
    Let $\cent:\Kc\to\R^d$ be a \textit{center function}, i.e.\ a measurable map compatible with translations and homogeneous under the scale action: 
    $
      \cent( t K+ \vect{x})=t\cent(K)+ \vect{x} ,
    $
    for any $ t \in (0,\infty) $ and any $ \vect{x} \in \R^d $.
    For example, $\cent$ can be the center of mass.
    We set
    $
      \mathcal{P}_{\cent}
      := \{ P\in\mathcal{P} : \cent (P)= \origin \} 
    $.
    Due to the natural homeomorphism
    \begin{eqnarray*}
      \Pc & \to & \R^d \times \Pcc\\
      P & \mapsto & \left( \cent (P), P - \cent(P) \right),
    \end{eqnarray*}
    we will consider from now on $X$ as a {germ-grain process} in $\R^d$ with {grain space} $\Pcc$, see \cite{SchneiderWeil08} for a definition of this concept.
    Since $\eta$ is stationary, this is also the case for $X$.
    This implies the existence of a probability measure $\Q$ on $\Pcc$ such that the intensity measure of the germ-grain process $X$ decomposes into $\Q$ and Lebesgue measure $\lambda_d$, that is
    \begin{equation}
      \label{eq:intSplit}
      \E\, X( \{ P-\cent(P) \in C\,,\,  \cent (P) \in A \}) =
      \IntMos  \lambda_d(A)\,  \Q(C)  \,,
    \end{equation}
    for $C \subset \Pcc$ and $A\subset\R^d$, where $ \IntMos   = \E\, X( \{ P \in \Pc\,,\, \cent (P) \in [0,1]^d \}) $.
    The probability measure $\Q$ and the constant $\IntMos$ are called, respectively, the {grain distribution} and the {intensity of $X$}.
    It is easy to see that $\IntMos$ is a multiple of $\gamma^d$, where $\gamma$ is the intensity of the Poisson hyperplane process.
    A random centred polytope $Z\in\Pcc$ with distribution $\Q$ is called a {typical cell} of~$X$.
    
   For a convex body $K$, we define its shape to be
    \[
      \sk(K) := \frac{1}{\Phi(K)} (K-\ck(K)) \,,
    \]
    i.e.\ its translated and normalized copy of center $\origin$ and $\Phi$-content $1$.
    We denote
    \[
      \Ks := \sk(\Kc) = \{ K\in\Kc : \ck(K) =\origin \,,\, \Phi(K)=1  \} \,,
    \]
    and similarly $\Pc_{\sk} := \sk(\Pc)$ and $\Pc_{n,\sk} := \sk(\Pc_n)$.
    
    The complementary theorem is a well-known result which shows how the shape and the $\Phi$-content are essential characteristics of the typical cell.
    It will be an important tool in the next section.
    This theorem holds in much more general settings.
    See \cite{Miles1971Complementary, MollerZuyev1996, Cowan2006, BaumstarkLast2009} for general versions of this theorem, or \cite{BonnetCalkaReitzner16} for a direct proof in the present setting.
   
    \begin{thm}[Complementary theorem]
      Let $ n \in \N $ be such that $ \P ( f(Z) = n ) > 0 $.
      If we condition the typical cell $Z$ to have $n$ facets, then
      \begin{enumerate}
        \item $ \Phi (Z) $ and $ \sk(Z) $ are independent random variables, 
        \item $ \Phi (Z) $ is $\Gamma_{\gamma, n-d}$ distributed, i.e.\ for any $ a \geq 0 $,
        $$\P ( \Phi (Z) \leq a ) = \frac{ \gamma^{n-d} }{(n-d-1)!} \int_0^a t^{ n - d - 1 } e^{ - \gamma t } \dint t \,.$$ 
      \end{enumerate}      
    \end{thm}
    
    We conclude this section by presenting standard notation for asymptotic approximations.
    For two given functions $f$ and $g$, we write $f(a)\sim g(a)$, as $a\to0$, if $\lim_{a\to0} \frac{f(a)}{g(a)} =1$; and $f(a)=o(g(a))$, as $a\to0$, if $\lim_{a\to0} \frac{f(a)}{g(a)} =0$.
    
  \section{General directional distribution}
  \label{sec:generalvarphi}
    Because of the second point of complementary theorem, the probability $\P(f(Z)=n \,,\, \Phi(Z)<a)$ is of order $a^{n-d}$ as $a\to0$.
    This implies that cells with small $\Phi$-content have $\nmin  = \min \{ n : \P(f(Z)=n)>0 \}$ facets with high probability.
    This fact combined with the first point of the complementary theorem gives us that $\sk(Z)$ conditioned on $\{\Phi(Z)<a\}$ converges in distribution, as $a\to0$, to $\sk(Z)$ conditioned on the event $\{f(Z)=\nmin \}$.
    The next theorem generalizes these properties to any size functional of the same order as $\Phi$, for example  the diameter, the mean-width, the first intrinsic volume or $\Radius$.
    It also gives a speed of convergence.
    In order to state the theorem we need to introduce measures on $\Kcs$ which are concentrated on the spaces $\Pns$. 
    We set $\mu_{n,\sk}(\cdot) := \P(f(Z)=n \,,\, \sk(Z)\in\cdot )$.
    This definition is motivated by the fact that, by the complementary theorem, we have for any Borel set $A\subset \Kc$,
    \begin{equation}\label{eq:muns}
      \P(f(Z)=n \,,\, Z\in A)
      = \frac{\gamma^{n-d}}{(n-d-1)!} \int_{\Pns} \int_0^\infty \1(tP\in A) e^{-\gamma t} t^{n-d-1} \,\dint t \,\dint \mu_{n,\sk}(P) \,.
    \end{equation}
    Recall that we say that a size functional $\Sigma$ is of degree $k$ when $\Sigma(tK)=t^k \Sigma(K)$, for any $t>0$ and $K\in\Kc$.
    For any size functional $\Sigma$ of degree $k$, we also consider a $\Sigma$-weighted version of $\mu_{n,\sk}$:
    \[
      \mu_{n,\sk,\Sigma}(\cdot)
      = \frac{1}{(n-d)!}\int_{\Pns} \1(P\in\cdot) \Sigma(P)^{-\frac{n-d}k} \,\dint \mu_{n,\sk}(P) \,.
    \]
    Note that $\mu_{n,\sk} = (n-d)!\, \mu_{n,\sk,\Phi}$.
    \begin{thm} \label{thm:generalvarphi}
      Assume that $\Sigma$ is a size functional homogeneous of degree $k$, and that there exists $c>0$ such that $c \, \Phi(K) \leq  \Sigma(K)^{\frac1k}$, for any $K\in\Kc$.
      Then, for any Borel set $S\subset\Kcs$ 
      \[
        \P\left(\sk(Z)\in S \mid \Sigma(Z)^{\frac1k}<a \right) \to \frac{\mu_{\nmin ,\sk,\Sigma}(S)}{\mu_{\nmin ,\sk,\Sigma}(\Kcs)} \,,
      \]
      as $a\to0$.
      In particular
      \[
        \P\left(f(Z)=\nmin  \mid \Sigma(Z)^{\frac1k}<a\right) \to 1 \,,
      \]
      as $a\to0$.
      Moreover, for any Borel set $ S \subset \Kcs $ for which $ \mu_{ \nmin , \sk , \Sigma } ( S ) > 0 $ and  $ \mu_{ \nmin + 1 , \sk , \Sigma } ( S ) > 0 $, it holds that
      \[
        \P\left(f(Z)>\nmin  \,,\, \sk(Z)\in S \mid \Sigma(Z)^{\frac1k}<a\right) \sim \gamma a \, \frac{\mu_{\nmin +1,\sk,\Sigma}(S)}{\mu_{\nmin ,\sk,\Sigma}(S)} \,,
      \]
      as $ a \to 0 $.
    \end{thm}
    
    In order to prove the theorem above, we need the following lemma.
    \begin{lem} \label{lem:limSigma}
      Let $S\subset \Kcs$ be a Borel set and $\Sigma$ be a size functional of degree $k$.
      We have
      \[
        (\gamma a)^{-(n-d)} \P\left(f(Z)=n \,,\, \sk(Z)\in S \,,\, \Sigma(Z)^{\frac1k}<a\right)
        \begin{cases}
          \to \mu_{n,\sk,\Sigma}(S) &\text{ as } a\to0 ,\\
          \leq \mu_{n,\sk,\Sigma}(S) &\text{ for any }a>0\,.
        \end{cases}
      \]
    \end{lem}
    \begin{proof}
      From \eqref{eq:muns}, we have
      \begin{align}
        \notag
        &\P \left(f(Z)=n \,,\, \sk(Z)\in S \,,\, \Sigma(Z)^{\frac1k}<a \right)
        \\ \notag
        &= \frac{\gamma^{n-d}}{(n-d-1)!} \int_{\Pns} \1(P\in S) \int_0^\infty \1\left(\Sigma(tP)^{\frac1k}<a\right) e^{-\gamma t} t^{n-d-1} \,\dint t \,\dint \mu_{n,\sk}(P)
        \\ \label{eq:revisedproof}
        &= \frac{\gamma^{n-d}}{(n-d-1)!} \int_{\Pns} \1(P\in S) \int_0^\infty \1\left(t<\frac{a}{\Sigma(P)^{\frac1k}}\right) e^{-\gamma t} t^{n-d-1} \,\dint t \,\dint \mu_{n,\sk}(P) \,.
      \end{align}
      Note that for any $P\in\Pns$, the substitution $s = \frac{ t }{a } $ gives that 
      \begin{align*}
        a^{-(n-d)} \int_0^\infty \1\left(t<\frac{a}{\Sigma(P)^{\frac1k}}\right) e^{-\gamma t} t^{n-d-1} \,\dint t 
        = \int_0^{\Sigma(P)^{-\frac{1}{k}}} e^{-\gamma a s } s^{n-d-1} \dint s \,,
      \end{align*}
      which converges monotonically to $ \frac{1}{n-d} \Sigma(P)^{-\frac{n-d}k} $, as $ a \to 0 $.
%
      Therefore multiplying both sides of \eqref{eq:revisedproof} by $(\gamma a)^{-(n-d)}$ and applying the monotone convergence theorem yields the assertion.
    \end{proof}
    
    We can now prove Theorem \ref{thm:generalvarphi}.    
    \begin{proof}[Proof of Theorem \ref{thm:generalvarphi}]
      In the case $|\{n : \P(f(Z)=n)>0\}|=1$, only the first part of the theorem matters, and it is a direct consequence of Lemma~\ref{lem:limSigma}.
      Therefore we can assume that  $|\{n : \P(f(Z)=n)>0\}|>1$.
      
      Because of Lemma~\ref{lem:limSigma}, we observe that for any $n$,
      \begin{equation}\label{eq:f>n}
        (\gamma a)^{-(n-d)} \P\left(f(Z)>n \,,\, \Sigma(Z)^{\frac1k}<a\right)
        \leq \sum_{k\geq 1} (\gamma a)^k  \mu_{n+k,\sk,\Sigma}(\Pc_{n+k,\sk}) \,,
      \end{equation}
      and we will see now that, because of the assumption made on $\Sigma$, the series above is finite and tends to $0$ as $a\to0$.
      Since $\Sigma(P)^{\frac1k} \geq c \,\Phi(P) = c$ for any $P\in\Pns$,
      \[
          \mu_{n,\sk,\Sigma}(\Pns)
          = \frac{1}{(n-d)!}\int_{\Pns} \Sigma(P)^{-\frac{n-d}k} \,\dint \mu_{n,\sk}(P) 
          \leq \frac{c^{-(n-d)}}{(n-d)!} \mu_{n,\sk}(\Pns) 
          \leq \frac{c^{-(n-d)}}{(n-d)!} \,.
      \]
      In particular \eqref{eq:f>n} gives
      \begin{align*}
        (\gamma a)^{-(n-d)} \P\left(f(Z)>n \,,\, \Sigma(Z)^{\frac1k}<a\right)
        &\leq \gamma a \sum_{k\geq 0} (\gamma a)^k  \frac{c^{-(n+k+1-d)}}{(n+k+1-d)!} \\
        &\leq \gamma a c^{-(n+1-d)} \exp (c^{-1}\gamma a) 
        \to 0 \,.
      \end{align*}      
      Thus, with Lemma \ref{lem:limSigma}, we have for any $n$ and $S\subset\Kcs$,
      \begin{align}
        \notag &(\gamma a)^{-(n-d)} \P\left(f(Z)\geq n \,,\, \sk(Z)\in S \,,\, \Sigma(Z)^{\frac1k}<a \right)
        \\& \label{sim} = (\gamma a)^{-(n-d)} \P\left(f(Z)=n \,,\, \sk(Z)\in S \,,\, \Sigma(Z)^{\frac1k}<a \right) +o(1) 
        \to \mu_{n,\sk,\Sigma}(S) \,.
      \end{align}
      Therefore
      \begin{align*}
        \P\left(\sk(Z)\in S \mid \Sigma(Z)^{\frac1k}<a \right)
        &= \frac{(\gamma a)^{-(n-d)}  \P\left(f(Z)\geq \nmin  \,,\, \sk(Z)\in S \,,\, \Sigma(Z)^{\frac1k}<a \right)}{(\gamma a)^{-(n-d)}  \P\left(f(Z)\geq \nmin  \,,\, \sk(Z)\in \Kcs \,,\, \Sigma(Z)^{\frac1k}<a \right)} \\
        & \to \frac{\mu_{\nmin ,\sk,\Sigma}(S)}{\mu_{\nmin ,\sk,\Sigma}(\Kcs)} \,,
      \end{align*}
      as $a\to0$.

      It only remains to prove the last point of the theorem.
      We consider a Borel set $ S \subset \Kcs $ for which $ \mu_{ \nmin , \sk , \Sigma } ( S ) > 0 $ and  $ \mu_{ \nmin + 1 , \sk , \Sigma } ( S ) > 0 $.
      Thus, with  \eqref{sim} , we obtain
      \begin{align*}
        &\P\left(f(Z)>\nmin  \,,\, \sk(Z)\in S \mid \Sigma(Z)^{\frac1k}<a\right) 
        \\ & = \gamma a \frac{ ( \gamma a )^{ - ( \nmin + 1 - d ) } \P\left(f(Z)\geq \nmin +1 \,,\, \sk(Z)\in S \,,\, \Sigma(Z)^{\frac1k}<a\right) }{ ( \gamma a )^{ - ( \nmin - d ) } \P\left(f(Z)\geq \nmin  \,,\, \sk(Z)\in\Kcs \,,\, \Sigma(Z)^{\frac1k}<a\right) }
        \\ & \sim \gamma a \, \frac{\mu_{\nmin +1,\sk,\Sigma}(S)}{\mu_{\nmin ,\sk,\Sigma}(S)} \,,      
      \end{align*}
      as $a\to0$.
      Note that we can write the asymptotic equivalence because the assumptions made on $ S $ imply that the right hand side is well defined and non zero.
    \end{proof}

  \section{Absolutely continuous directional distribution}
    \label{sec:abscont}
    In this section we assume that $\varphi$ is absolutely continuous with respect to the spherical Lebesgue measure.
    Almost surely each cell of the tessellation contains a unique ball of maximal volume called the inball of the cell.
    Thus, we can assume the centering function $\ck$ to be the center of the inball.
    In the first subsection we show that cells with small $\Sigma$-content are simplices of a random shape.
    In the second subsection we investigate the speed of convergence of $\P(f(Z)>d+1 \mid \Sigma(Z)<a)$ to $0$, as $a\to0$.
     
    \subsection{Absolutely continuous case: Cells with small \texorpdfstring{$\Sigma$}{Sigma}-content are simplices with random shape}
      \sectionmark{Absolute continuous case}
      
      Theorem 10.4.6 of \cite{SchneiderWeil08} describes the distribution of the typical cell of a stationary isotropic hyperplane tessellation.
      As observed after the proof of the theorem in \cite{SchneiderWeil08}, it easily extends to the non isotropic case if $\varphi$ is absolutely continuous with respect to the spherical Lebesgue measure.
      It states that for any Borel set $A\subset\Kc$,
      \begin{equation} \label{eq:Calka} 
        \begin{split}
          \P(Z\in A)
          &= \frac{\gamma^{d+1}}{(d+1) \gamma^{(d)}}
          \int_{\Pstraight} \int_0^\infty 
          e^{-\gamma r}
          \P \left( \bigcap_{H\in \eta \cap \Fc^{r\B}} H^- \cap (r \simplex(\vvectu)) \in A \right) 
          \dint r 
          \\& \qquad \times \Delta_d(\vvectu) \, \dint \varphi^{d+1} (\vvectu) \,,
        \end{split} 
      \end{equation}
      where $\eta \cap \Fc^{r\B}$ denotes the set of hyperplanes of the process $\eta$ which do not intersect the ball $r\B$, 
      \[
        \vvectu = (u_0,\ldots,u_d) \,,
      \]
      \begin{equation} \label{eq:defP}
        \Pstraight 
        = \left\{ (\vvectu)\in(\S^{d-1})^{d+1} : \vect u_0,\ldots , \vect u_d \text{ are not all in one closed half sphere} \right\} \,,
      \end{equation}
      and
      \begin{equation} \label{eq:Deltad}
        \simplex(\vvectu) = \bigcap_{i=0}^d H(\vect u_i,1)^- ,
        \ \text{ and } \ 
        \Delta_d(\vvectu) = \lambda_d ( \mathrm{ConvexHull}(\vect u_0,\ldots,\vect u_d)) \,.
      \end{equation}

      Note that the factor $ e^{-\gamma r} $ in the integrand above is the probability that no hyperplanes of $ \eta $ intersect the ball $ r B^d $.
      Therefore for any measurable set of simplices $ A \subset \Pc_{d+1} $, any $r>0$, and any $(d+1)$-tuple $ \vvectu \in \Pstraight $, the integrand of \eqref{eq:Calka} is equal to
      $ \1 ( r \simplex(\vvectu) \in A ) \P ( \eta \cap r T (\vvectu) = \emptyset ) = \1 ( r \simplex(\vvectu) \in A ) e^{-\gamma r \Phi(T (\vvectu))}$.
      In particular, for an open Borel set $S\subset \Pc_{\sk}$ and $a>0$, Equation \eqref{eq:Calka} gives
      \begin{align}
        \notag
        &\P\left( f(Z) = d+1 \,,\, 
        \sk(Z)\in S \,,\, 
        \Sigma(Z)^{\frac1k} < a \right)\\
        &\notag= \frac{\gamma^{d+1}}{(d+1) \gamma^{(d)}}
        \int_{\Pstraight} \int_0^\infty 
        e^{-\gamma r \Phi(\simplex(\vvectu))}
        \1\left( \sk(\simplex(\vvectu))\in S \right)
        \1\left( \Sigma(r\simplex(\vvectu))^{\frac1k} < a  \right)
        \dint r 
        \\\notag&\qquad \times \Delta_d(\vvectu) \, \dint \varphi^{d+1} (\vvectu) 
        \\\label{eq:shaped+1}
        &\sim \gamma a \cdot \frac{\gamma^{d}}{(d+1) \gamma^{(d)}}
        \int_{\Pstraight}
        \frac{\Delta_d(\vvectu) }{\Sigma(\simplex(\vvectu))^{\frac1k} }
        \1\left( \sk(\simplex(\vvectu))\in S \right)
        \dint \varphi^{d+1} (\vvectu) \,,
      \end{align}  
      as $a\to 0$.
      
      But, because of \eqref{ineq:sizes}, 
      \[
        \P\left(f(Z)>d+1 \,,\, \Sigma(Z)^{\frac1k}<a\right)
        \leq\P\left(f(Z)>d+1 \,,\, \radius (Z)<\frac{a}{\Sigma(\B)^{\frac1k}} \right) \,,
      \]
      and thus \eqref{eq:Calka} also gives
      \begin{align*}
        &\P \left(f(Z)>d+1 \,,\, \Sigma(Z)^{\frac1k}<a\right)
        \\&\leq \frac{\gamma^{d+1}}{(d+1) \gamma^{(d)}}
        \int_{\Pstraight} \int_0^\infty 
        e^{-\gamma r}
        \P \left(\text{it exists $ H \in \eta $ such that $ H \cap r T ( u_{0:d} ) \neq \emptyset$ } \right)
        \\&\qquad  \times \1\left( r<\frac{a}{\Sigma(\B)^{\frac1k}}  \right)
        \dint r \, \Delta_d(\vvectu) \, \dint \varphi^{d+1} (\vvectu) \,.
      \end{align*}       
      Observe that, for any fixed $\vvectu\in\Pstraight$, the probability in the integrand is decreasing and tends to $0$, as $r\to0$.
      This implies
      \begin{align*}
        &\P \left(f(Z)>d+1 \,,\, \Sigma(Z)^{\frac1k}<a\right)
        \\\notag&\leq \gamma a  \frac{\gamma^{d}}{(d+1) \gamma^{(d)} \Sigma(\B)^{\frac1k}}
        \int_{\Pstraight} 
        \P \left(\text{it exists $ H \in \eta $ such that $ H \cap \frac{a}{\Sigma(\B)^{\frac1k}} T ( u_{0:d} ) \neq \emptyset$ } \right)
        \\&\notag\qquad  \times 
        \Delta_d(\vvectu) \, \dint \varphi^{d+1} (\vvectu) \,.
      \end{align*}
      With the dominated convergence theorem we get that the last integral tends to $0$, as $ a \to 0$.
      Therefore
      \begin{align}
        \P \left(f(Z)>d+1 \,,\, \Sigma(Z)^{\frac1k}<a\right)
        \label{eq:mored+1}&
        =o(\gamma a)\,,
      \end{align} 
      as $a\to0$.
            
      Equations \eqref{eq:shaped+1} and \eqref{eq:mored+1} immediately give the following theorem.
      \begin{thm}\label{them:abscont1}
        If $\varphi$ is absolutely continuous with respect to the spherical Lebesgue measure, and $\Sigma$ is a size functional of degree $k$, then for any Borel set $S\subset\Pc_{\sk}$, we have
        \[
          \P\left(f(Z)=d+1 \mid \Sigma(Z)^{\frac1k}<a\right) \to 1 \,,
        \]
        and
        \[
          \P\left(\sk(Z)\in S \mid \Sigma(Z)^{\frac1k}<a\right)
          \to \frac{\xi_{\varphi,\Sigma}(S)}{\xi_{\varphi,\Sigma}(\Pc_{\sk})} \,,
        \]
        as $a\to0$, where
        \[
          \xi_{\varphi,\Sigma}(S)
          = \int_{\Pstraight}
          \frac{\Delta_d(\vvectu) }{\Sigma(\simplex(\vvectu))^{\frac1k} }
          \1\left( \sk(\simplex(\vvectu))\in S \right)
          \dint r  
          \, \dint \varphi^{d+1} (\vvectu) \,.
        \]
      \end{thm}
 
      
    \subsection{Absolutely continuous case: Speed of convergence}
      \sectionmark{Absolute continuous case}
      Our goal now is to find how fast $\P(f(Z) = d+1 \mid \Sigma(Z)^{\frac1k}<a)$ tends to $1$ as $a\to0$.
      Observe that there exists a strictly positive constant $c$ depending on $\varphi$ and $\Sigma$ such that
      \begin{equation}\label{sim2}
        \P\left(\Sigma(Z)^{\frac1k}<a\right) 
        \sim \P\left(f(Z)=d+1\,,\,\Sigma(Z)^{\frac1k}<a\right) 
        \sim c \gamma a \,.
      \end{equation}
      The first equivalence is due to Theorem \ref{them:abscont1}.
      The second one is a special case of $\eqref{eq:shaped+1}$.
      Thus, our problem boils down to study how fast ${\P(f(Z) > d+1 \,,\, \Sigma(Z)^{\frac1k}<a)}$ tends to $0$, as $a\to0$.
      
      Because of \eqref{ineq:sizes}, we have
      \begin{align}\label{eq:SigmaInradius}
        \P(f(Z) > d+1  \,,\, \Sigma(Z)^{\frac1k}<a) &
        \begin{cases}
          \leq \P \left( f(Z) > d+1  \,,\, \radius(Z) < \frac{a}{\Sigma(\B)^{\frac1k}} \right) \\
          \geq \P \left( f(Z) > d+1  \,,\, \Radius(Z) < \frac{a}{\Sigma(\B)^{\frac1k}} \right) \,.
        \end{cases}         
      \end{align}
      This shows that it is essential to study the cases of the inradius and the circumradius.
      Size functionals of the same order as the circumradius were already considered in Section \ref{sec:generalvarphi}, and by essentially the same technics, see in particular \eqref{sim}, there exists a positive constant $c$ such that
      \begin{equation*}
        \P \left( f(Z) > d+1  \,,\, \Radius(Z) < \frac{a}{\Sigma(\B)^{\frac1k}} \right)
        \sim c (\gamma a)^2 \,.
      \end{equation*}
      Combining this with \eqref{sim2} and \eqref{eq:SigmaInradius} leads to
      \begin{align}\label{eq:lowerbound}
        \P\left(f(Z) > d+1 \mid \Sigma(Z)^{\frac1k}<a\right)
        \geq \frac{\P\left( f(Z) > d+1  \,,\, \Radius(Z) < \frac{a}{\Sigma(\B)^{\frac1k}} \right)}{\P\left(\Sigma(Z)^{\frac1k}<a\right)}
        \sim c \gamma a \,.
      \end{align}
      
      It remains to give an upper bound for $\P \left( f(Z) > d+1 \,,\, \radius(Z) < \frac{a}{\Sigma(\B)^{\frac1k}} \right)$, which turns out to be technically more involved and is the essential part of this subsection.
      By \eqref{eq:Calka}, we have
      \begin{align*}
        & \frac{(d+1)}{\gamma} \frac{\gamma^{(d)}}{\gamma^d} \, \P(f(Z) > d+1 \,,\, \radius(Z)<a )
        \\&= \int_{\Pstraight} \int_0^a e^{-\gamma r} \left( 1 - e^{-\gamma r (\Phi(\simplex(\vvectu))-\Phi(\B) )} \right) \dint r \, \Delta_d(\vvectu) \, \dint \varphi^{d+1} (\vvectu) 
        \\&\leq \int_{\Pstraight} \int_0^a 1 - e^{-\gamma r \Phi(\simplex(\vvectu))} \dint r \, \Delta_d(\vvectu) \, \dint \varphi^{d+1} (\vvectu) \,,
      \end{align*}
      where $\Pstraight$, $\simplex(\vvectu)$ and $\Delta_d(\vvectu)$ are defined as in \eqref{eq:defP} and \eqref{eq:Deltad}.
      It is tempting to upper bound the integrand $1 - e^{-\gamma r \Phi(\simplex(\vvectu))}$ by $\gamma r \Phi(\simplex(\vvectu))$, since for $r\to0$ these quantities are equivalent.
      This would lead us to
      \begin{align*}
        &\P(f(Z) > d+1 \,,\, \radius(Z)<a )
        \\&\leq (\gamma a)^2 \frac1{2(d+1)} \frac{\gamma^d}{\gamma^{(d)}} \int_{\Pstraight} \Phi(\simplex(\vvectu)) \Delta_d(\vvectu) \, \dint \varphi^{d+1} (\vvectu) \,,
      \end{align*}
      but this is useless since the integral $\int_{\Pstraight} \Phi(\simplex(\vvectu)) \Delta_d(\vvectu) \, \dint \varphi^{d+1} (\vvectu)$ turns out to diverge.
      We need to consider more carefully the contribution of simplices $\simplex(\vvectu)$ which have a big $\Phi$-content.
      The factor $\Delta_d(\vvectu)$ is not important to get the order of the integral.
      We upper bound it by $\Delta_{\max} = \max_{\vvectu\in \Pstraight} \Delta_d(\vvectu)$.
      So we have
      \begin{align}  \label{eq:UBPfr}
        \begin{split}
          &\frac{(d+1)}{\gamma \Delta_{\max}} \frac{\gamma^{(d)}}{\gamma^d} \, \P(f(Z) > d+1  \,,\, \radius(Z)<a )
          \\&\leq
          \int_{\Pstraight} \int_0^a 1 - e^{-\gamma r \Phi(\simplex(\vvectu))} \dint r \, \dint \varphi^{d+1} (\vvectu)
        \end{split}      
      \end{align}
      In order to go further, we need to prove the following key lemma.
      \begin{lem}
        \label{lem:upboundPhiofDelta}
        Assume that $\varphi$ has a bounded density with respect to the spherical Lebesgue measure.
        There exists a constant $C_\varphi$, only depending on $\varphi$, such that for any increasing function $f\colon \R_+ \to \R_+ $, we have
        \[
          \int_\Pstraight f(\Phi(\simplex(\vvectu))) \, \dint \varphi^{d+1}(\vvectu)
          \leq C_\varphi \int_1^\infty f(t) \frac1{t^2} \,\dint t \,.
        \]
      \end{lem}
      
      \begin{proof}
        By standard continuity arguments, there exists a constant $ c_0 > 1 $, depending on $ \varphi $ such that $ \Phi ( \simplex ( \vvectu ) ) > c_0 $ for any $ \vvectu \in \Pstraight $.
        To see this, one must extend the map $ \Pstraight \ni \vvectu \mapsto \Phi ( T ( \vvectu ) ) \in [ 0 , \infty ) $ to the compact domain $ ( \S^{ d - 1 } )^{ d + 1 } $ and codomain $ [ 0 , \infty ] $ by defining $  \Phi ( T ( \vvectu ) ) = \infty $ for any $ \vvectu \in ( \S^{ d - 1 } )^{ d + 1 } \setminus \Pstraight $.
        The extended map is continuous and since its domain is compact and $  \Phi ( T ( \vvectu ) ) >  \Phi ( \B ) = 1 $ for any  $ \vvectu \in ( \S^{ d - 1 } )^{ d + 1 } $, there exists a constant $ c_0 $ as described above.
        Therefore we can assume without loss of generality that $ f ( t ) = 0 $ for any $ t \in [ 0 , c_0 ] $.
        
        The increasing function $ f $ can be uniformly approximated by stair functions of the form $ \sum_{ i \in \N }  a_i \1_{ [ t_i , \infty ) } $, where $ a_i \geq 0 $, and $ t_i \geq c_0 $.
        Hence we only need to show that the lemma holds for indicator functions of the form $ \1_{ [ t_0 , \infty ) } $, where $ t_0 \geq c_0 $.
        
        Set $ c > 0 $ such that $ \varphi \leq c \sigma $, where $ \sigma $ denotes the spherical Lebesgue measure.
        We have
        \begin{align*}
          \int_\Pstraight \1_{ [ t_0 , \infty ) } (\Phi(\simplex(\vvectu))) \, \dint \varphi^{d+1}(\vvectu)
          & \leq c^{ d + 1 } \int_\Pstraight \1_{ [ t_0 , \infty ) } (\Phi(\simplex(\vvectu))) \, \dint \sigma^{d+1}(\vvectu) \, .
        \end{align*}
        
        Since $ \Phi : \Kc \to \R $ is increasing with respect to the inclusion and the $ \Phi $-content of a ball is equal to its radius, we have that any set of $ \Phi $-content greater than $ t_0 $ is not included in the ball $ t_0 \B $.
        In particular
        \begin{align*}
          \int_\Pstraight \1_{ [ t_0 , \infty ) } (\Phi(\simplex(\vvectu))) \, \dint \varphi^{d+1}(\vvectu)
          &  \leq c^{ d + 1 } \int_\Pstraight \1 ( \simplex ( \vvectu ) \not\subset t_0 \B ) \, \dint \sigma^{d+1}(\vvectu) \, .
        \end{align*}
        
        For any $\vvectu\in\Pstraight$ and $i\in\{0,\ldots,d\}$, we consider 
        \begin{equation*}
          \vect v(\vvectu , i)
          := \bigcap_{j\in\{ 0, \ldots d \} \setminus \{i\}} H(\vect u_j,1) \,.
        \end{equation*}
        This is the vertex of the simplex $\simplex(\vvectu)$ which is not contained in the face with outward normal vector $\vect u_i$, see Figure \ref{fig:triangle}.
        The condition $ \simplex ( \vvectu ) \not\subset t_0 \B $ is equivalent to saying that the vertex of $ \simplex ( \vvectu ) $ the furthest away from the origin has norm greater than $ t_0 $.
        Thus
        \begin{align*}
          & \int_\Pstraight \1_{ [ t_0 , \infty ) } (\Phi(\simplex(\vvectu))) \, \dint \varphi^{d+1}(\vvectu)
          \\ &  \leq c' \int_\Pstraight \1 ( \simplex ( \vvectu ) \not\subset t_0 \B ) \1 ( \| v ( \vvectu , 0 ) \| \geq \| v ( \vvectu , i ) \| \text{ for any $ i \in [ d ] $ } ) \, \dint \sigma^{d+1}(\vvectu)
          \\ & \leq c' \int_\Pstraight \1 ( \| v ( \vvectu , 0 ) \| \geq t_0 ) \, \dint \sigma^{d+1}(\vvectu) \, ,
        \end{align*}
        where $ c' := c^{ d + 1 } ( d +1 ) $.
        
        Observe that $ \| v ( \vvectu , 0 ) \| \geq t_0 $ if and only if the hyperplane containing the unit vectors $ \vect{ u_1 } , \ldots , \vect{ u_d } $ is at distance less than or equal to $ \frac{ 1 }{ t_0 } $, see Figure \ref{fig:triangle}.
        We denote this distance by $ D ( \origin , \mathrm{aff} ( u_1 , \ldots , u_d) ) $.
        \begin{figure}          
          \begin{tikzpicture}
            \def\r{1.9} 
            \def\a{76} 
            \def\b{165} 
            \pgfmathsetmacro\px{\r/cos(\a)}
            \pgfmathsetmacro\qy{\r/cos(\b-90)}
            \pgfmathsetmacro\t{\r*cos(\a)}
            
            \begin{scope}[overlay]
              \coordinate (p) at (\px,0);
              \coordinate (p1) at (2*\a:\px);
              \coordinate (p2) at (-2*\a:\px);
              \coordinate (q1) at (0,\qy);
              \coordinate (q2) at (2*\b-90:\qy);
          
              \path[name path=lineup] (p)-- (p1);
              \path[name path=linedown] (p)-- (p2);
              \path[name path=lineleft] (q1)-- (q2);
          
              \path[name intersections={of=lineup and lineleft,by={r1}}];
              \path[name intersections={of=linedown and lineleft,by={r2}}];
            \end{scope}
        
            \draw (0,0) node[below left]{$\origin$} circle [radius=\r];
            
            \draw (0,0) -- (\b:\r) node[left]{$\vect u_0$};
              \filldraw (\b:\r) circle (1.5pt);
            \draw (0,0) -- (\a:\r) node[above right]{$\vect u_1$};
              \filldraw (\a:\r) circle (1.5pt);
            \draw (0,0) -- (-\a:\r) node[below right]{$\vect u_2$};
              \filldraw (-\a:\r) circle (1.5pt);
            
            \filldraw (90:\r) circle (1.5pt) node[above]{$\vect u'_1$};
            \filldraw (-90:\r) circle (1.5pt) node[below]{$\vect u'_2$};
            \filldraw (0:\r) circle (1.5pt) node[below right]{$\vect v$};
            
            \draw[red,very thick] (p) -- (r1) -- (r2) -- cycle;
            
            \filldraw (p) circle (1.5pt) node[right,black]{$\vect v(\vvectu,0)$};
            \filldraw (r1) circle (1.5pt) node[above,black]{$\vect v(\vvectu,2)$};
            \filldraw (r2) circle (1.5pt) node[below,black]{$\vect v(\vvectu,1)$};
            
            \draw[dashed] (0,0) -- (p);
            \draw[dashed] (0,-\r) -- (0,\r);
            \draw[dashed] (\a:\r) -- (-\a:\r);
            \draw (\t,0) node[below right,black]{$t$};
          \end{tikzpicture}
          \caption{Construction of $\simplex(\vvectu)$ (thick triangle).}
          \label{fig:triangle}
        \end{figure}
        We have
        \begin{align*}
          & \int_\Pstraight \1_{ [ t_0 , \infty ) } (\Phi(\simplex(\vvectu))) \, \dint \varphi^{d+1}(\vvectu)
          \\ &  \leq c' \int_\Pstraight \1 \left( D ( \origin , \mathrm{aff} ( u_1 , \ldots , u_d) ) \leq \frac{ 1 }{ t_0 } \right) \, \dint \sigma^{d+1}(\vvectu) \, .
        \end{align*}
        
        Note that the integrand does not involve $ u_0 $.
        Now, we release the constraint $ \vvectu \in \Pstraight $ and integrate over the variable $ u_0 $, which does not appear in the integrand.
        This gives
        \begin{align*}
          & \int_\Pstraight \1_{ [ t_0 , \infty ) } (\Phi(\simplex(\vvectu))) \, \dint \varphi^{d+1}(\vvectu)
          \\ &  \leq c'' \int_{ ( \S^{ d - 1 } )^d } \1 \left( D ( \origin , \mathrm{aff} ( u_1 , \ldots , u_d) ) \leq \frac{ 1 }{ t_0 } \right) \, \dint \sigma^d ( \vvectuone ) \, ,
        \end{align*}
        where $ c'' := c' \sigma ( \S^{d-1} ) $.
        
        Using an integral transformation due to Miles, see Theorem 4 of \cite{Miles1971} or Note 6 on page 286 of \cite{SchneiderWeil08} we have
        \begin{align*}
          & \int_\Pstraight \1_{ [ t_0 , \infty ) } (\Phi(\simplex(\vvectu))) \, \dint \varphi^{d+1}(\vvectu)
          \\ & \leq c'' (d-1)! \int_{\S^{d-1}} \int_0^1 \int_{(H(\vect v,t)\cap\,\S^{d-1})^d} \1\left( t \leq \frac{ 1 }{ t_0 } \right) \, \Delta_{d-1}(\vvectuone) \dint {\sigma'_{\vect v,t}}^d (\vvectuone)
          \frac {\dint t}{(1-t^2)^{\frac{d}2}} \dint \sigma(\vect v) \, ,
        \end{align*}
        where $\sigma'_{\vect v,t}$ denotes the spherical Lebesgue measure on the $(d-2)$-dimensional sphere $H(\vect v,t)\cap\S^{d-1}$.
        For any $ \vect v \in \S^{d-1} $ and $ t \in (0,1) $, considering the diffeomorphism 
        \begin{eqnarray*}
          	H(\vect v,0)\cap\,\S^{d-1} & \to & H(\vect v,t)\cap\,\S^{d-1}
          	\\
          	\vect u' & \mapsto & \vect u = t \vect v + \sqrt{1-t^2} \vect u' \,,
        \end{eqnarray*}  
        we get
        \begin{align*}
          &\int_{(H(\vect v,t) \cap\S^{d-1})^d} \Delta_{d-1}(\vvectuone) \, \dint (\sigma'_{\vect v,t})^d (\vvectuone)
          \\&= \int_{ ( H(\vect v,0) \cap \S^{d-1} )^d }
              (1-t^2)^{\frac{d-1}{2}} \Delta_{d-1}(\vvectuone') \left( (1-t^2)^{\frac{d-2}2} \right)^{d}   
              \, \dint {\sigma'_{\vect v,0}}^d (\vvectuone') \, .
        \end{align*}
        Therefore
        \begin{align*}
          & \int_\Pstraight \1_{ [ t_0 , \infty ) } (\Phi(\simplex(\vvectu))) \, \dint \varphi^{d+1}(\vvectu)
          \\ & \leq c'' (d-1)! \int_{\S^{d-1}} \int_0^{ \frac{ 1 }{ t_0 }} \int_{(H(\vect v,0)\cap\,\S^{d-1})^d} \Delta_{d-1}(\vvectuone) \, \dint {\sigma'_{\vect v,0}}^d (\vvectuone)
          ( 1 - t^2 )^{ \frac{ d^2 - 2 d - 1 }{ 2 }} \dint t \, \dint \sigma(\vect v) \, .
        \end{align*}
        
        Recall that $ t_0 \geq c_0 > 1 $, and therefore the factor $ ( 1 - t^2 )^{ \frac{ d^2 - 2 d - 1 }{ 2 }} $ can be uniformly bounded on the interval $ [ 0 , \frac{ 1 }{ t_0 } ] $.
        For an appropriate constant $ c^{(3)} $ this implies
        \begin{align*}
          & \int_\Pstraight \1_{ [ t_0 , \infty ) } (\Phi(\simplex(\vvectu))) \, \dint \varphi^{d+1}(\vvectu)
          \\ & \leq c^{(3)} \int_{\S^{d-1}} \int_0^{ \frac{ 1 }{ t_0 }} \int_{(H(\vect v,0)\cap\,\S^{d-1})^d} \Delta_{d-1}(\vvectuone) \, \dint {\sigma'_{\vect v,0}}^d (\vvectuone) \dint t \, \dint \sigma(\vect v) \, . 
        \end{align*}
        Thus by setting 
        \[ 
          C_\varphi 
          := c^{(3)} \int_{\S^{d-1}} \int_{(H(\vect v,0)\cap\,\S^{d-1})^d} \Delta_{d-1}(\vvectuone) \, \dint {\sigma'_{\vect v,0}}^d (\vvectuone) \, \dint \sigma(\vect v)  \, ,
        \]
        we obtain that
        \begin{align*}
          \int_\Pstraight \1_{ [ t_0 , \infty ) } (\Phi(\simplex(\vvectu))) \, \dint \varphi^{d+1}(\vvectu)
          & \leq C_\varphi \frac{ 1 }{ t_ 0 }
          = C_\varphi \int_0^\infty \1_{ [ t_0 , \infty ) } \frac{ 1 }{ t^2 } \dint t \, ,
        \end{align*}
        which is what we needed to show.
      \end{proof}

      Now that we proved Lemma \ref{lem:upboundPhiofDelta}, we go back to our original problem which is to derive an upper bound for $\P(f(Z) > d+1 \,,\, \radius(Z)<a )$.
      Equation \eqref{eq:UBPfr} and Lemma~\ref{lem:upboundPhiofDelta} give
      \begin{align*}
        \frac{(d+1)}{\gamma \Delta_{\max}} \frac{\gamma^{(d)}}{\gamma^d} \, \P(f(Z) > d+1  \,,\, \radius(Z)<a )
        &\leq C_\varphi \int_1^\infty \int_0^a 1 - e^{-\gamma r t} \dint r \, \frac{1}{t^2} \dint t
        \\&= C_\varphi \left( I_1 + I_2 + I_3 \right) \,,
      \end{align*}
      where
      \begin{align*}
        I_1 
        = \int_1^{\frac1{\gamma a}} \int_0^a \frac{1-e^{-\gamma r t}}{t^2}  \dint r \dint t
        \leq \int_1^{\frac1{\gamma a}} \int_0^a \frac{\gamma r t}{t^2}  \dint r \dint t
        = \frac{\gamma a^2}{2} \ln \left( \frac1{\gamma a} \right) \,,
      \end{align*}
      \begin{align*}
        I_2 
        = \int_{\frac1{\gamma a}}^\infty \int_0^{\frac1{\gamma t}} \frac{1-e^{-\gamma r t}}{t^2}  \dint r \dint t
        \leq \int_{\frac1{\gamma a}}^\infty \int_0^{\frac1{\gamma t}} \frac{\gamma r t}{t^2}  \dint r \dint t
        = \frac{\gamma a^2}{4} \,,
      \end{align*}
      and
      \begin{align*}
        I_3 
        &= \int_{\frac1{\gamma a}}^\infty \int_{\frac1{\gamma t}}^a \frac{1-e^{-\gamma r t}}{t^2}  \dint r \dint t
        = \int_0^a \int_{\frac1{\gamma r}}^\infty \frac{1-e^{-\gamma r t}}{t^2} \dint t \dint r
        = \left( \int_{1}^\infty \frac{1-e^{-t}}{t^2} \dint t \right) \frac{\gamma a^2}{2}
        \leq \frac{\gamma a^2}{2} \,.
      \end{align*}
      Therefore, by setting 
      $
        C'_\varphi
        := C_\varphi \left( \frac12 +\frac14 + \frac12 \right) \frac{\Delta_{\max} \gamma^d}{(d+1)\gamma^{(d+1)}}
      $,
      we have
      \begin{align*}
        \P(f(Z) > d+1  \,,\, \radius(Z)<a )
        \leq C'_\varphi (\gamma a)^2 \ln \left( \frac1{\gamma a} \right) \,,
      \end{align*}
      for $a< e^{-1} \gamma^{-1}$.
      And with \eqref{eq:SigmaInradius}, we obtain
      \begin{align*}
        \P \left( f(Z) > d+1  \,,\, \Sigma(Z)^{\frac1k}<a \right)
        &\leq \P \left( f(Z) > d+1  \,,\, \radius(Z) < \frac{a}{\Sigma(\B)^{\frac1k}} \right) \\
        &\leq C'_\varphi (\gamma \Sigma(\B)^{-\frac1k} a)^2 \ln \left( \frac1{\gamma \Sigma(\B)^{-\frac1k} a} \right) \,,
      \end{align*}
      for $a < \Sigma(\B)^{-\frac1k} e^{-1} \gamma^{-1} $.
      This implies the existence of a constant 
      $
        C_{\varphi,\Sigma}
      $, depending only on $\varphi$ and $\Sigma$, such that
      \begin{align}
        \label{eq:jointNotsimplexSigma}
        \P \left( f(Z) > d+1  \,,\, \Sigma(Z)^{\frac1k} <a \right)
        \leq C_{\varphi,\Sigma} (\gamma a)^2 \ln \left( \frac1{\gamma a} \right) \,,
      \end{align}
      for $a <\frac{1}{2\gamma}$.
      
      Now, we easily get the following theorem.
      \begin{thm}\label{thm:speed}
        Assume that $\varphi$ has a bounded density with respect to the spherical Lebesgue measure
        and that $\Sigma$ is a size functional of degree $k>0$.
        There exist positive constants $c , c' >0$, depending only on $\varphi$ and $\Sigma$, such that
        \[
          c \gamma a 
          \leq \P\left(f(Z)>d+1 \mid \Sigma(Z)^{\frac1k} <a \right) 
          \leq c' \gamma a \ln \left( \frac1{\gamma a} \right) \,,
        \]
      for $a <\frac{1}{2\gamma}$.
        Moreover the lower bound is sharp if $\Sigma=\Radius$. 
      \end{thm}
      \begin{proof}
        The lower bound is \eqref{eq:lowerbound}, which also includes the equality case if $\Sigma=\Radius$.
        Thus we only need to show the upper bound. 
        Equation \eqref{eq:shaped+1} says that
        \[
          \P \left( f(Z)= d+1 \,,\, \Sigma(Z)^{\frac1k} <a \right) 
          \sim c_1 \gamma a \,,
        \]
        as $a\to0$, where $c_1>0$ is a constant depending only on $\varphi$ and $\Sigma$.
        In particular, it exists a constant $c_2>0$, depending only on $\varphi$ and $\Sigma$, such that
        \[
          \P \left( \Sigma(Z)^{\frac1k} <a \right)
          \geq \P \left( f(Z) = d+1 \,,\, \Sigma(Z)^{\frac1k} <a \right) 
          \geq c_2 \gamma a \,,
        \]
        for any $a<\gamma^{-1}$.
        Thus, with \eqref{eq:jointNotsimplexSigma}, we get
        \begin{align*}
          \P \left( f(Z)>d+1 \mid \Sigma(Z)^{\frac1k} <a \right) 
          &= \frac{\P\left( f(Z)>d+1 \,,\, \Sigma(Z)^{\frac1k} <a \right)}{ \P\left(\Sigma(Z)^{\frac1k} <a \right)}
          \\&\leq \frac{C_{\varphi,\Sigma} \left(\gamma a\right)^2 \ln \left( \frac1{\gamma a} \right)}{c_2 \gamma a}
          \\&= \frac{C_{\varphi,\Sigma}}{c_2} \gamma a \ln \left( \frac1{\gamma a} \right) ,
        \end{align*}
        for $a< \gamma^{-1}$.
        This yields the theorem.
      \end{proof}

  \section{Directional distribution with atoms}
    \label{sec:atoms}
    In the two previous sections we have seen that $\P\left(f(Z)=\nmin  \mid \Sigma(Z)^{\frac1k}<a\right)\to1$, as $a\to0$, if 
    \begin{itemize}
      \item either $\Sigma^{\frac1k}>c\Phi$ for some $c>0$,
      \item or $\varphi$ is absolutely continuous with respect to the spherical Lebesgue measure.
    \end{itemize}
    In contrast to these results we present in the following theorem a class of directional distributions for which cells with small inradius do not have with high probability the minimal number of facets.
    This class includes, for example, the normalization of the sum of the spherical Lebesgue measure and a non-zero discrete measure.
    In the planar case, other examples are discrete measures with support of cardinality at least $2\times 3$.
    \begin{thm}\label{them:atoms}
      Assume that $\varphi$ is such that:
      \begin{enumerate}
        \item the support of $\varphi$ includes $d+1$ distinct points which are not all contained in some half sphere,
        \item $\varphi$ has an atom at $u\in\S^{d-1}$, i.e.\ $\varphi(\{u,-u\})>0$.
      \end{enumerate}
      Then
      \[
        \P\left(f(Z)=\nmin  \mid \radius(Z)<a\right)\not\to1 , 
        \text{ as } a\to 0,
      \]
      meaning that, conditionally on the event $\{\radius(Z)<a\}$, the probability that $Z$ has the minimal number of facets possible does not tend to $1$, as $a\to0$.
    \end{thm}
    \begin{proof}
      The first condition is equivalent to $\nmin=d+1$.
      For $a>0$, set
      \[
        A_{u,a} = \{ P\in\Pc : \text{$P$ has two parallel facets orthogonal to $u$ with distance less than $a$} \}\,.
      \]
      Since a polytope which has a couple of parallel facets cannot be a simplex, we have that
      \begin{equation}\label{eq:atom1}
        \P\left(f(Z)>d+1  \mid \radius(Z)<a\right) 
        \geq \P\left(Z\in A_{u,a}  \mid \radius(Z)<a\right) \,.
      \end{equation}
      But Theorem 10.4.8 of \cite{SchneiderWeil08} (where $\gamma = 2 \widehat{\gamma}$) tells us that $\P(\radius(Z)<a)=1-e^{-\gamma a}\sim \gamma a$ as $a\to0$, and we will show below that, under condition 2,
      \begin{equation}\label{eqtoshow}
        \begin{split}
          &\text{there exists $c>0$ and $a_0>0$ such that}\\
          &\P(Z\in A_{u,a}) > c \gamma a, \text{ for any }a\in(0,\gamma^{-1} a_0) . 
        \end{split}
      \end{equation}       
      Therefore, for $a\in(0,\gamma^{-1} a_0)$,
      \begin{equation*}
        \P\left(Z\in A_{u,a}  \mid \radius(Z)<a\right) 
        = \frac{\P\left(Z\in A_{u,a}\right) }{\P\left(\radius(Z)<a\right) }
        > \frac{c \gamma a}{\gamma a} 
        = c > 0 \,.
      \end{equation*}
      Combined with \eqref{eq:atom1} and the fact that $\nmin=d+1$, it implies the theorem.
      
      \medskip
      It remains to show that \eqref{eqtoshow} holds.
      We will use Theorem 10.4.7 of \cite{SchneiderWeil08} which gives an explicit representation of the typical cell as an intersection of the zero cell $Z_{\origin}$ and $d+1$ halfspaces containing the origin in their boundaries.
      For $\xi,u_1,\ldots,u_d\in\S^{d-1}$ in general position, we denote by $T_\xi(u_1,\ldots,u_d)$ the simplicial cone whose facets are subsets of the hyperplanes $H(u_1,0),\ldots,H(u_d,0)$ and for which the  origin $\origin$ is the unique highest point in direction $\xi$, meaning that $0=\langle \xi , \origin \rangle > \langle \xi , x \rangle $ for any $x\in T_\xi(u_1,\ldots,u_d)\setminus\{\origin\}$.
      With Theorem 10.4.7 of \cite{SchneiderWeil08}, we get that there exist a probability measure $\phi_d$ on $(\S^{d-1})^d$ and a vector $\xi\in\S^{d-1}$ such that, for any $a>0$,
      \begin{align*}
        \P(Z\in A_{u,a})
        &= \int_{(\S^{d-1})^d} \P(Z_{\origin} \cap T_\xi(u_1,\ldots,u_d)\in A_{u,a}) \, \dint \phi_d(u_1,\ldots,u_d)
        \\&\geq 
        \int_{(\S^{d-1})^d} \1(u_1=u) \P(Z_{\origin}\cap T_\xi(u_1,\ldots,u_d)\in A_{u,a}) \, \dint \phi_d(u_1,\ldots,u_d) \,.
      \end{align*}
      Observe that for any fixed $u_2,\ldots,u_d$, and $a\to0$, we have
      \begin{align*}
        \P(Z_{\origin} \cap T_\xi(u,u_2\ldots,u_d)\in A_{u,a})
        &\sim \P(\exists t\in (0,a) : H(u,t)\in \eta)
        = 1-e^{-\gamma a \varphi(\{u\})}
        \\&\sim \gamma a \varphi(\{u\}) \,.
      \end{align*}
      Thus, by setting a constant $c$ such that
      \begin{align*}
        0<c
        &< 
         \varphi(\{u\}) \int_{(S^{d-1})^d} \1(u_1=u) \, \dint \phi_d(u_1,\ldots,u_d) \,, 
      \end{align*}      
      we have  $ \P(Z\in A_{u,a}) > c \gamma a$ for $a$ small enough.
      This is precisely \eqref{eqtoshow}.
    \end{proof}
      
    
  \bibliography{Bibliography.bib}  
  
\end{document}